\documentclass[12pt]{amsart}

\usepackage[cp1251]{inputenc}
\usepackage{amsthm,amssymb,amsmath,amsfonts,wasysym}
\usepackage{textcomp} 
\usepackage{etoolbox}
\usepackage{comment}

\usepackage{graphicx}
\usepackage[matrix,arrow,curve]{xy}
\usepackage{longtable}
\newcommand{\mumu}{{\boldsymbol{\mu}}}

\newcommand{\ka}{\ensuremath{\Bbbk}}

\newcommand{\kka}{\ensuremath{\overline{\Bbbk}}}

\newcommand{\Pro}{\ensuremath{\mathbb{P}}}

\newcommand{\Aut}{\ensuremath{\operatorname{Aut}}}

\newcommand{\Gal}{\ensuremath{\operatorname{Gal}}}

\makeatletter
\@addtoreset{equation}{section}
\makeatother

\newtheorem{theorem}[equation]{Theorem}
\newtheorem{proposition}[equation]{Proposition}
\newtheorem{lemma}[equation]{Lemma}

\theoremstyle{definition}

\theoremstyle{remark}
\newtheorem{remark}[equation]{Remark}

\title{Quotients of Severi--Brauer surfaces}
\address{Steklov Mathematical Institute of Russian Academy of Sciences, 8 Gubkina st., Moscow, 119991, Russia}
\address{Laboratory of Algebraic Geometry, National Research University Higher School of Economics, 6 Usacheva str., Moscow 119048, Russia}
\email{trepalin@mccme.ru}
\author{Andrey Trepalin}

\begin{document}

\begin{abstract}
We show that a quotient of a non-trivial Severi--Brauer surface $S$ over arbitrary field $\ka$ of characteristic $0$ by a finite group $G \subset \Aut(S)$ is $\ka$-rational, if and only if $|G|$ is divisible by $3$. Otherwise, the quotient is birationally equivalent to $S$.

\end{abstract}



\maketitle


Let $\ka$ be an arbitrary field of characteristic zero, and $\kka$ be its algebraic closure. \mbox{A~$d$-dimensional} variety~$X$ is called a \textit{Severi--Brauer variety} if $\overline{X} = X \otimes \kka$ is isomorphic to $\Pro^d_{\kka}$. If $d = 1$ then $X$ is a conic, and if $d = 2$ then $X$ is called a \textit{Severi--Brauer surface}.

A Severi--Brauer variety is called \textit{non-trivial} if it is not isomorphic to $\Pro^d_{\ka}$. It is well known that a Severi--Brauer variety is trivial if and only if it has a $\ka$-point. Moreover, if~$X$ is a non-trivial $d$-dimensional Severi--Brauer variety and $d + 1$ is a prime number, then degree of any point on $X$ is divisible by $d + 1$ (see \cite[Theorem 53]{Kol16}). 

Complete classification of finite subgroups of automorphism groups of non-trivial Severi--Brauer surfaces is obtained in works of C.\,Shramov and V.\,Vologodsky \cite{ShV20}, \cite{Sh20a}, \cite{Sh20b}, \cite{Sh21}. Let $\mumu_n$ be a cyclic group of order $n$. Then we have the following.

\begin{theorem}[cf. {\cite[Theorem 1.3(ii)]{Sh21}}]
\label{class2}
Let $S$ be a non-trivial Severi--Brauer surface over a field $\ka$ of characteristic zero. Then any finite subgroup of $\Aut(S)$ is isomorphic to~$\mumu_n$, $\mumu_{3n}$, $\mumu_n \rtimes \mumu_3$ or $\mumu_3 \times \left( \mumu_n \rtimes \mumu_3 \right)$, where $n$ is a positive integer divisible only by~primes congruent to $1$ modulo $3$ (including the case $n=1$), and $\mumu_n \rtimes \mumu_3$ is a semidirect product corresponding to an outer automorphism of $\mumu_n$ of order $3$ acting non-trivially on each non-trivial element of $\mumu_n$.

Moreover, for any group $G$ mentioned above there exists a field $\ka$ of charectiristic zero and a non-trivial Severi--Brauer surface $S$ over $\ka$, such that $G \subset \Aut(S)$.

\end{theorem}

The aim of this paper is to obtain a birational classification of quotients of non-trivial Severi--Brauer surfaces by finite subgroups of automorphism groups. In particular, we want to~answer a question, for which Severi--Brauer surfaces $S$ and finite subgroups \mbox{$G \subset \Aut(S)$} the quotient $S / G$ is $\ka$-rational (i.e. birationally equivalent to $\Pro^2_{\ka}$). Note that for an~algebraically closed field $\kka$ of characteristic zero quotients of $\kka$-rational surfaces by finite automorphism groups are always $\kka$-rational by Castelnuovo's rationality criterion (see~\cite{Cast94}). Moreover, for arbitrary field $\ka$ of characteristic zero quotients of $\Pro^2_{\ka}$ by~finite automorphism groups are always $\ka$-rational (see \cite[Theorem 1.3]{Tr14}). One can find more results about rationality of quotients of $\ka$-rational surfaces by finite automorphism groups in \cite{Tr19}. The main result of this paper is the following theorem.

\begin{theorem}
\label{quot3}
Let $S$ be a non-trivial Severi--Brauer surface over a field $\ka$ of characteristic zero, and $G$ be a finite subgroup of $\Aut(S)$. Then the quotient $S / G$ is $\ka$-rational, if and only if $|G|$ is divisible by $3$. Otherwise, the quotient $S / G$ is birationally equivalent to~$S$.
\end{theorem}

This theorem is an analogue of the following proposition.

\begin{proposition}
\label{quot2}
Let $C$ be a conic over a field $\ka$ of characteristic zero without $\ka$-points, and $G$ be a finite subgroup of $\Aut(C)$. Then the quotient $C / G$ is isomorphic to $\Pro^1_{\ka}$, if and only if $G$ has even order.
\end{proposition}

\begin{proof}
The degree of any point on $C$ is even. Assume that the order of $G$ is odd. Then any $G$-orbit on the set of geometric points on $C$ has odd cardinality. Thus the image of~any such orbit on $C / G$ is not defined over $\ka$, and the quotient $C / G$ does not have $\ka$-points and is not isomorphic to $\Pro^1_{\ka}$.

Now assume that the order of $G$ is even. Therefore there is an element $g \in G$ of order~$2$. The anticanonical map given by the linear system $|-K_C|$ gives an embedding $C \hookrightarrow \Pro^2_{\ka}$, such that the action of $g$ on $C$ induces an action of $g$ on $\Pro^2_{\ka}$. One can choose coordinates $(x : y : z)$ on $\Pro^2_{\ka}$ such that the action of $g$ is given by $(x:y:z) \mapsto (x: y : -z)$. Then the quotient map $C \rightarrow C / \langle g \rangle$ is just a restriction of the projection $\Pro^2_{\ka} \dashrightarrow \Pro^1_{\ka}$ given by~$(x:y:z) \mapsto (x:y)$.

Therefore some $\langle g \rangle$-orbits are defined over $\ka$, and thus some $G$-orbits are defined over~$\ka$ too. It means that there are $\ka$-points on the quotient $C / G$, and therefore this quotient is~isomorphic to $\Pro^1_{\ka}$, since one has $\Pro^1_{\kka} / G \cong \Pro^1_{\kka}$ for any finite subgroup $G \subset \Aut\left(\Pro^1_{\kka}\right)$.

\end{proof}

The proof of Theorem \ref{quot3} is more complicated, and we decompose it into several lemmas.

\begin{lemma}
\label{action}
Let $S$ be a non-trivial Severi--Brauer surface over a field $\ka$ of characteristic zero and $G$ be a finite subgroup of~$\Aut(S)$ isomorphic to $\mumu_n$. Then the set of $G$-fixed points on $S$ consists of three isolated geometric points. If $n > 3$ then these points are defined over extension $K / \ka$ of degree $3$. 
\end{lemma}

\begin{proof}
Consider the action of $G$ on $\overline{S} \cong \Pro^2_{\kka}$. This action can be diagonalized in $\mathrm{PGL}_3\left(\kka\right)$. One can see that the set of fixed points of $G$ either consists of an isolated fixed point and a~line, or consists of three isolated points $p_1$, $p_2$ and $p_3$. In the former case the isolated fixed point is unique, and therefore defined over~$\ka$, but there are no $\ka$-points on~$S$. Therefore $G$ has three isolated fixed geometric points on $S$.

Isolated fixed geometric points of $G$ are transitevely permuted by~the~Galois group~$\Gal\left(\kka / \ka\right)$. Therefore these points are defined either over extension $K / \ka$ of~degree~$3$ with~$\Gal\left(K / \ka \right) \cong \mumu_3$, or over extension $K / \ka$ of degree $6$ with $\Gal\left(K / \ka \right) \cong \mathfrak{S}_3$, where~$\mathfrak{S}_3$ is a~non-abelian group of order $6$. In particular, there is an element $\gamma \in \Gal\left(\kka / \ka\right)$ of~order~$3$, such that $\gamma p_1 = p_2$, $\gamma p_2 = p_3$, $\gamma p_3 = p_1$.

If $n > 3$ then there is a subgroup $\mumu_{p} \subset G$ for a prime number $p$ such that $p$ equals $1$ modulo~$3$. The~action of a generator of $\mumu_p$ on $\overline{S} \cong \Pro^2_{\kka}$ can be written as $\operatorname{diag}(\xi_p; \xi_p^a; 1)$, where $\xi_p$ is \mbox{a~$p$-th} root of unity. This element acts on the tangent spaces $T_{p_1} \overline{S}$, $T_{p_2} \overline{S}$ and~$T_{p_3} \overline{S}$ as $\operatorname{diag}\left(\xi_p; \xi_p^a \right)$, $\operatorname{diag}\left(\xi_p^{a-1}; \xi_p^{-1} \right)$ and $\operatorname{diag}\left(\xi_p^{-a}; \xi_p^{1-a} \right)$ respectively.

One has $\gamma p_1 = p_2$. Thus $\gamma \left( \xi_p \right) = \xi_p^{a-1}$, and
$$
\xi_p^{-1} = \gamma \left( \xi_p^a \right) = \gamma \left( \xi_p \right)^a= \xi_p^{a(a-1)}. 
$$
Therefore $a^2 - a + 1 = 0$ modulo $p$. In particular, $a \neq 1$ and $a \neq -1$ modulo $p$.

Assume that the Galois group $\Gal\left(\kka / \ka\right)$ contains an element $\delta$ such that $\delta p_1 = p_1$, $\delta p_2 = p_3$ and $\delta p_3 = p_2$. Then $\delta \left( \xi_p \right) = \xi_p^a$, and $\delta \left( \xi_p^2 \right) = \xi_p$. Therefore $a^2 = 1$ modulo $p$. But $a \neq \pm 1$ modulo $p$. We have a contradiction. Hence if $n > 3$ then the points $p_1$, $p_2$ and $p_3$ are defined over an extension $K / \ka$ of degree $3$.
\end{proof}

\begin{lemma}
\label{nonrat}
Let $S$ be a non-trivial Severi--Brauer surface over a field $\ka$ of characteristic zero and $G$ be a finite subgroup of~$\Aut(S)$ isomorphic to $\mumu_n$, where $n$ is a positive integer divisible only by primes congruent to $1$ modulo $3$. Then the quotient $S / G$ is birationally equivalent to $S$.

\end{lemma}

\begin{proof}
Any point on a non-trivial Severi--Brauer surface has degree divisible by $3$. Therefore there are no $\ka$-points on $S / G$, since cardinalities \mbox{of~$G$-orbits} are not divisible by~$3$.  

For $n = 1$ the assertion of the lemma is trivial. Therefore assume that $n > 1$.

By Lemma \ref{action} the set of $G$-fixed points on $S$ consists of three isolated geometric points defined over an extension $K / \ka$ of degree $3$. Any cyclic group acting on $\overline{S} \cong \Pro^2_{\kka}$ is~a~subgroup of a torus acting on $\overline{S}$. Therefore the quotient $S / G$ is a $\ka$-form of a toric surface with three singular points, that are images of $G$-fixed points on $S$. The Galois group $\mumu_3 \cong \Gal\left(K/\ka\right)$ acts on the fan of the corresponding toric surface. Let us resolve the singularities of $S / G$ and run $\mumu_3$-equivariant minimal model program. As~a~result we should obtain $\mumu_3$-minimal del Pezzo surface or conic bundle $S'$, such that $\overline{S}'$ is a toric surface. Any toric del Pezzo surface or conic bundle $X$ has $K_X^2 \geqslant 6$, and one can easily check that for such surface a~group $\mumu_3$ can faithfully act only on the fan of a del Pezzo surface of~degree $9$ or $6$. Moreover, del~Pezzo surface of degree $6$ is not $\mumu_3$-minimal, since one can $\mumu_3$-equivariantly blow down a triple of disjoint $(-1)$-curves. Thus $S'$ is a del Pezzo surface of degree~$9$ without $\ka$-points, i.e. non-trivial Severi--Brauer surface.

The existence of rational map $S \dashrightarrow S'$ implies that the class of $S'$ in the Brauer group~$\mathrm{Br}\left(\ka\right)$ lies in the cyclic group generated by the class of $S$ (see \cite[Exercise 3.3.8(iii)]{GS18}). Therefore $S'$ is either isomorphic to $\Pro^2_{\ka}$, that is impossible since there are no $\ka$-points on $S'$, or~isomorphic to $S$, or~isomorphic to $S^{op}$, where $S^{op}$ is the Severi--Brauer surface corresponding to the central simple algebra opposite to the one corresponding to~$S$. The surface $S^{op}$ is birationally equivalent to $S$ (see \cite{Ami55}). Thus $S / G$ is birationally equivalent to $S$.

\end{proof}

\begin{remark}
\label{eqnonrat}
Note that the proof of Lemma \ref{nonrat} can be easily generalized on the case where $N \cong \mumu_n$ is a normal subgroup of a finite group $G$, such that $G$ is isomorphic to $\mumu_{3n}$, \mbox{$\mumu_n \rtimes \mumu_3$} or~$\mumu_3 \times \left( \mumu_n \rtimes \mumu_3 \right)$. In this case the quotient $S / N$ is $G / N$-birationally equivalent to~a~Severi--Brauer surface $S'$, since the action of $G / N$ on the fan corresponding to $S / N$ is either trivial, or coincides with the action of $\mumu_3 \cong \Gal\left(K / \ka\right)$.
\end{remark}

Now assume that a group $\mumu_3$ acts on a non-trivial Severi--Brauer surface $S$. Then by Lemma \ref{action} the group $\mumu_3$ has three isolated fixed geometric points defined either over extension $K / \ka$ of degree $3$ with $\Gal\left(K / \ka \right) \cong \mumu_3$, or over extension $K / \ka$ of degree $6$ with $\Gal\left(K / \ka \right) \cong \mathfrak{S}_3$. We consider these two cases separately, and show that in the both cases the quotient $S / \mumu_3$ is $\ka$-rational.

\begin{lemma}
\label{rat3}
Let $S$ be a non-trivial Severi--Brauer surface and $G$ be a finite subgroup of~$\Aut(S)$ isomorphic to $\mumu_3$. Assume that the isolated fixed points of $G$ are defined over extension $K / \ka$ of degree $3$ with $\Gal\left(K / \ka \right) \cong \mumu_3$. Then the quotient $S / G$ is $\ka$-rational.

\end{lemma}

\begin{proof}
Let $p_1$, $p_2$ and $p_3$ be the isolated fixed geometric points of $G$ on $S$. One can \mbox{$G$-equivariantly} blow up this triple of points and get a del Pezzo surface $\widetilde{S}$ of degree~$6$. Note that each of six $(-1)$-curves is $G$-invariant, therefore the six points of~intersection of~these six curves are $G$-fixed. One can easily check that on the tangent space of $\overline{\widetilde{S}}$ at~each of these points the group $G$ acts as $\langle \operatorname{diag}\left(\omega; \omega\right) \rangle$, where $\omega$ is a third root of unity.

The group $\Gal\left(K / \ka \right) \cong \mumu_3$ does not preserve any $G$-fixed point on $\widetilde{S}$. Therefore one can blow up any triple of $G$-fixed points defined over $\ka$, and get a surface $\widetilde{X}$. One has $K_{\widetilde{X}}^2 = 3$, the exceptional divisor of the blowup $\widetilde{X} \rightarrow \widetilde{S}$ is a triple of pointwisely $G$-fixed $(-1)$-curves $E_1$, $E_2$ and $E_3$, and the proper transforms of the six $(-1)$-curves on $\widetilde{S}$ are $(-2)$-curves on~$\widetilde{X}$. One can check that $\widetilde{X}$ is a weak del Pezzo surface, and the anticanonical map
$$
\varphi_{|-K_{\widetilde{X}}|}: \widetilde{X} \rightarrow X
$$
\noindent blows down the six $(-2)$-curves into three $A_2$-singularities. Moreover, $X$ is a singular cubic surface, and the images of $E_1$, $E_2$ and $E_3$ are lines passing through pairs of singular points. 

The linear systems $|-K_{\widetilde{S}}|$, $|-K_{\widetilde{X}}|$ and $|-K_X|$ are $G$-invariant, since the morphisms $\widetilde{S} \rightarrow S$, $\widetilde{X} \rightarrow \widetilde{S}$ and $\widetilde{X} \rightarrow X$ are $G$-equivariant. Therefore the action of $G$ on $X$ induces an~action of $G$ on $\Pro^3_{\ka}$. Moreover, $G$ pointwisely fixes a plane in $\Pro^3_{\ka}$ since it fixes three lines meeting each other. Therefore one can choose coordinates $(x : y : z : t)$ on $\Pro^3_{\ka}$ such that the~action of $G$ is given by $(x:y:z:t) \mapsto (x: y : z : \omega t)$. Then the quotient map $X \rightarrow X / G$ is just a restriction of the projection $\Pro^3_{\ka} \dashrightarrow \Pro^2_{\ka}$ given by
$$
(x:y:z:t) \mapsto (x:y:z)
$$
\noindent (cf. the proof of Proposition \ref{quot2}). Therefore $X / G \cong \Pro^2_{\ka}$, and $S / G$ is $\ka$-rational, since $S / G$ and $X / G$ are birationally equivalent.

\end{proof}

For the remaining case $\Gal\left(K / \ka \right) \cong \mathfrak{S}_3$ we need the following lemma.

\begin{lemma}
\label{dP6rat}
Let $S'$ be a del Pezzo surface of degree $6$ over arbitrary field $\ka$ of characteristic~$0$. If $S'$ contains a point of degree $2$ and a point degree $3$, then $S'$ is $\ka$-rational.
\end{lemma}

\begin{proof}
Let $L$ be an extension of degree $2$ of $\ka$, such that there is a $L$-point on $S' \otimes L$. Then~$S' \otimes L$ is $L$-rational, therefore we can find a pair of geometric points $p_1$ and $p_2$ on $S'$ defined over $L$ and not lying on $(-1)$-curves.

Let us consider the blowup $X \rightarrow S'$ at $p_1$ and $p_2$. One can check that the surface $X$ is~a~(possibly, weak) del Pezzo surface of degree $4$, and there is a point of degree $3$ on $X$. The anticanonical map given by $|-K_X|$ maps $X$ to $X' \subset \Pro^4_\ka$, where $X'$ is a (possibly, singular) del Pezzo surface of degree $4$. The arguments of the proof of \cite[Lemma~2.4]{Sh20c} show that in this case $X'$ has a smooth $\ka$-point. For convenience of the reader, we give the other proof of this result.

Indeed, let $q_1$, $q_2$ and $q_3$ be three geometric points on $X'$, such that this triple is~defined over $\ka$. Such triple exists since $S'$ contains a point of degree $3$. Consider a plane $\Pi$ defined over $\ka$ and passing through $q_1$, $q_2$ and $q_3$. The intersection number $\Pi \cdot X' = 4$, and the multiplicity of the intersection is the same at~the~points $q_1$, $q_2$ and $q_3$. Therefore, if $\Pi \cap X'$ is zero-dimensional, then there is an other smooth $\ka$-point $q$ on $\Pi \cap X'$.

If $\Pi \cap X'$ is one-dimensional, then the intersection $C = X' \cap \Pi$ is a conic, since $X'$ is an intersection of two quadrics in $\Pro^4_{\ka}$. If $C$ is smooth, then there are infinitely many $\ka$-points on $C$, since there is a point of degree $3$ on $C$. If $C$ is a union of two lines then the points $q_1$, $q_2$ and $q_3$ have to lie on one of these lines. Therefore this line is defined over~$\ka$. If~$C$ is a double line, then this line is defined over $\ka$. In each case we find infinitely many $\ka$-points on $C$, thus there is a smooth $\ka$-point on $X'$.

The image of any $\ka$-point on $X'$ under the map $X' \rightarrow S'$ is a $\ka$-point on $S'$. Therefore~$S'$ is $\ka$-rational by \cite[Chapter 4]{Isk96}.



\end{proof}

\begin{lemma}
\label{rat6}
Let $S$ be a non-trivial Severi--Brauer surface and $G$ be a finite subgroup of~$\Aut(S)$ isomorphic to $\mumu_3$. Assume that the isolated fixed points of $G$ are defined over extension $K / \ka$ of degree $6$ with $\Gal\left(K / \ka \right) \cong \mathfrak{S}_3$. Then the quotient $S / G$ is $\ka$-rational.

\end{lemma}

\begin{proof}

The group $G$ has three isolated fixed points on $S$, that are defined over $K$. Therefore the quotient $S / G$ is a $\ka$-form of a singular toric surface, such that the three singular points are not defined over $\ka$, and defined over $K$. The Galois group $\mathfrak{S}_3 \cong \Gal\left(K/\ka\right)$ acts on the fan of the corresponding toric surface. Let us resolve the singularities of $S / G$ and run $\mathfrak{S}_3$-equivariant minimal model programm. As a result we should obtain $\mathfrak{S}_3$-minimal del Pezzo surface or conic bundle $S'$, such that $\overline{S}'$ is a toric surface. Any toric del Pezzo surface or conic bundle $X$ has $K_X^2 \geqslant 6$, and one can easily check that for such surface a~group $\mathfrak{S}_3$ can faithfully act only on the fan of a del Pezzo surface of degree $9$ or $6$.

Let us show that $S'$ is $\ka$-rational. Note that there is a normal subgroup $\mumu_3$ in~the~group $\mathfrak{S}_3 \cong \Gal\left(K/\ka\right)$, therefore there exist an extension $L = K^{\mumu_3}$ of $\ka$ of degree~$2$, such that the isolated fixed points of $G$ are not defined over $L$, and defined over the extension $K / L$ of degree $3$. Thus $S' \otimes L \cong \left(S \otimes L\right) / G$ is $L$-rational by Lemma \ref{rat3}. So there is a point of~degree $2$ on $S'$.  If $S'$ is a del Pezzo surface of degree $9$ then $S'$ is $\ka$-rational, since the~line passing through a point of~degree $2$ is defined over $\ka$.

If $S'$ is a del Pezzo surface of degree $6$ then there is a point of degree $3$ on $S'$, since the image of any point of degree $3$ on $S$ under the quotient map $S / G$ is either a $\ka$-point, or~a~point of degree $3$. Therefore $S'$ is $\ka$-rational by Lemma \ref{dP6rat}.

\end{proof}

Now we can prove Theorem \ref{quot3}.

\begin{proof}[Proof of Theorem \ref{quot3}]
A finite subgroup $G$ of $\Aut(S)$ is isomorphic to $\mumu_n$, $\mumu_{3n}$, $\mumu_n \rtimes \mumu_3$ \mbox{or~$\mumu_3 \times \left( \mumu_n \rtimes \mumu_3 \right)$} by Theorem \ref{class2}.

If there is a subgroup $N \cong \mumu_n$ in $G$, where $n$ is a positive integer divisible only by~primes congruent to $1$ modulo $3$, then $N$ is normal in $G$, and the quotient $S / N$ \mbox{is $G / N$-birationally} equivalent to a non-trivial Severi--Brauer surface $S'$ with the action of $G / N$ by Lemma~\ref{nonrat} and Remark~\ref{eqnonrat}. In particular, if $G \cong \mumu_n$, then $S / G$ is birationally equivalent to $S$ by~Lemma~\ref{nonrat}.

The other three cases of $G$ are reduced to the case, where a group $\mumu_3$ or $\mumu_3^2$ acts on~a~non-trivial Severi--Brauer surface, since $S / G$ and $S' / (G / N)$ are birationally equivalent. 

The quotient of Severi--Brauer surface $S$ by $\mumu_3$ is $\ka$-rational by Lemmas \ref{rat3} and \ref{rat6}, and~the~quotient $S / \mumu_3^2$ is birationally equivalent to a quotient of a toric surface $S / \mumu_3$ by~a~cyclic group $\mumu_3$. This quotient is $\ka$-rational by \cite[Proposition 4.5]{Tr14}.

\end{proof}

\smallskip
\textbf{Acknowledgements.} The author is grateful to Costya Shramov and Sergey Gorchinskiy for many useful discussions and comments. 


\bibliographystyle{alpha}

\begin{thebibliography}{XX}




\bibitem[Ami55]{Ami55}
S.\,A.\,Amitsur,
\newblock Generic splitting fields of central simple algebras, Ann. Math., 62:2 (1955), 8--43





\bibitem[Cast94]{Cast94}
G.\,Castelnuovo,
\newblock Sulla razionalit\`{a} delle involuzioni piane,
\newblock Math. Ann., 1894, 44, 125--155










\bibitem[GS18]{GS18}
S.\,Gorchinskiy, C.\,Shramov,
\newblock Unramified Brauer group and its applications,
\newblock Translations of Mathematical Monographs, 246, American Mathematical Society, Providence, 2018, xvii+179 pp.




\bibitem[Isk96]{Isk96}
V.\,A.\,Iskovskikh,
\newblock Factorization of birational mappings of rational surfaces from the point of view of Mori theory,
\newblock Uspekhi Mat. Nauk, 1996, 51, 3--72 (in Russian); translation in Russian Math. Surveys, 1996, 51, 585--652



\bibitem[Kol16]{Kol16}
J.\,Koll\'{a}r,
\newblock Severi--Brauer varieties; a geometric treatment,
\newblock preprint, arXiv:1606.04368 (2016)







\bibitem[Sh20a]{Sh20a}
C.\,A.\,Shramov,
\newblock Birational automorphisms of Severi--Brauer surfaces,
\newblock Sb. Math., 211:3 (2020), 466--480

\bibitem[Sh20b]{Sh20b}
C.\,A.\,Shramov,
\newblock Non-abelian groups acting on Severi--Brauer surfaces,
\newblock Mat. Zametki, 108:6 (2020), 916--917

\bibitem[Sh20c]{Sh20c}
C. Shramov,
\newblock Automorphisms of cubic surfaces without points,
\newblock Int. J. Math., 31:11 (2020), 2050083, 15 pp.

\bibitem[Sh21]{Sh21}
C. Shramov,
\newblock Finite groups acting on Severi--Brauer surfaces,
\newblock Eur. J. Math., 7:2 (2021), 591--612

\bibitem[ShV20]{ShV20}
C.\,Shramov, V.\,Vologodsky,
\newblock Boundedness for finite subgroups of linear algebraic groups,
\newblock preprint, arXiv:2009.14485 (2020) 

\bibitem[Tr14]{Tr14}
A.\,S.\,Trepalin,
\newblock Rationality of the quotient of $\mathbb{P}^2$ by finite group of automorphisms over arbitrary field of characteristic zero,
\newblock Cent. Eur. J. Math., 12:2 (2014), 229--239




\bibitem[Tr19]{Tr19}
A.\,Trepalin,
\newblock Quotients of del Pezzo surfaces,
\newblock Int. J. Math., 30:11 (2019), 1950068, 40 pp.


\end{thebibliography}

\end{document}